\documentclass[12pt]{amsart}
\usepackage{amscd,amsmath,amsthm,amssymb}
\usepackage[left]{lineno}
\usepackage{color}
\usepackage{stmaryrd}
\usepackage[utf8]{inputenc}
\usepackage{epstopdf}
\usepackage{graphicx}
\definecolor{verylight}{gray}{0.97}
\definecolor{light}{gray}{0.9}
\definecolor{medium}{gray}{0.85}
\definecolor{dark}{gray}{0.6}

 \def\frk{\mathfrak}               % font for "Fraktur"

 \def\mm{{\frk m}}

 \def\G{{\mathcal G}}

 \def\P{{\mathcal P}}

  \def\xb{{\mathbf x}}

  \def\Soc{{\mathbf Soc}}

 \def\opn#1#2{\def#1{\operatorname{#2}}} % to make operators
 \opn\chara{char} \opn\length{\ell} \opn\pd{pd} \opn\rk{rk}
 \opn\projdim{proj\,dim} \opn\injdim{inj\,dim} \opn\rank{rank}
 \opn\depth{depth} \opn\grade{grade} \opn\height{height}
 \opn\embdim{emb\,dim} \opn\codim{codim}
 
 \opn\Tr{Tr} \opn\bigrank{big\,rank}
 \opn\superheight{superheight}\opn\lcm{lcm}
 \opn\trdeg{tr\,deg}%\emph{
 \opn\reg{reg} \opn\lreg{lreg} \opn\ini{in} \opn\lpd{lpd}
 \opn\mi{mi}
 \opn\size{size} \opn\sdepth{sdepth}
 \opn\link{link}\opn\fdepth{fdepth}\opn\lex{lex}
 \opn\tr{tr}
 \opn\type{type}
 \opn\gap{gap}
 \opn\arithdeg{arith-deg}
 \opn\HS{HS}
 \opn\mgrade{m-grade}
 \opn\div{div} \opn\Div{Div} \opn\cl{cl} \opn\Cl{Cl}
 \opn\Spec{Spec} \opn\Supp{Supp} \opn\supp{supp} \opn\Sing{Sing}
 \opn\Ass{Ass} \opn\Min{Min}\opn\Mon{Mon}
 \opn\Ann{Ann} \opn\Rad{Rad} \opn\Soc{Soc}
 \opn\Im{Im} \opn\Ker{Ker} \opn\Coker{Coker} \opn\Am{Am}
 \opn\Hom{Hom} \opn\Tor{Tor} \opn\Ext{Ext} \opn\End{End}
 \opn\Aut{Aut} \opn\id{id}
 
 \opn\nat{nat}
 \opn\pff{pf}%   \pf exists already
 \opn\Pf{Pf} \opn\GL{GL} \opn\SL{SL} \opn\mod{mod} \opn\ord{ord}
 \opn\Gin{Gin} \opn\Hilb{Hilb}\opn\sort{sort}
 \opn\PF{PF}\opn\Ap{Ap}
 \opn\mult{mult}
 \opn\bight{bight}
 \opn\a{a}
 \opn\Ap{Ap}
 \opn\aff{aff}
 \opn\relint{relint} \opn\st{st}
 \opn\lk{lk} \opn\cn{cn} \opn\core{core} \opn\vol{vol}  \opn\inp{inp} \opn\nilpot{nilpot}
 \opn\link{link} \opn\star{star}\opn\lex{lex}\opn\set{set}
 \opn\width{wd}
 \opn\Fr{F}
 \opn\QF{QF}
 \opn\G{G}
 \opn\type{type}\opn\res{res}
 \opn\conv{conv}
 \opn\Ind{Ind}
 \opn\cone{cone}
 \opn\im{im}
 \opn\gr{gr}
 
 \def\pot#1#2{#1[\kern-0.28ex[#2]\kern-0.28ex]}

 \opn\dirlim{\underrightarrow{\lim}}
 \opn\inivlim{\underleftarrow{\lim}}
 \let\union=\cup

 \let\iso=\cong
 \let\Union=\bigcup

 \let\to=\rightarrow
 
 \def\Implies{\ifmmode\Longrightarrow \else
         \unskip${}\Longrightarrow{}$\ignorespaces\fi}
 \def\implies{\ifmmode\Rightarrow \else
         \unskip${}\Rightarrow{}$\ignorespaces\fi}
 \def\iff{\ifmmode\Longleftrightarrow \else
         \unskip${}\Longleftrightarrow{}$\ignorespaces\fi}

 \let\:=\colon
 \newtheorem{Theorem}{Theorem}[section]
 
 \newtheorem{Corollary}[Theorem]{Corollary}
 \newtheorem{Proposition}[Theorem]{Proposition}
 \newtheorem{Remark}[Theorem]{Remark}

 \newtheorem{Question}[Theorem]{Question}
 
 \let\epsilon\varepsilon
 \let\kappa=\varkappa
 \def\qed{\ifhmode\textqed\fi
       \ifmmode\ifinner\quad\qedsymbol\else\dispqed\fi\fi}
 \def\textqed{\unskip\nobreak\penalty50
        \hskip2em\hbox{}\nobreak\hfil\qedsymbol
        \parfillskip=0pt \finalhyphendemerits=0}
 \def\dispqed{\rlap{\qquad\qedsymbol}}
 
 \opn\dis{dis}
 \def\pnt{{\raise0.5mm\hbox{\large\bf.}}}
 
 \opn\Lex{Lex}

\begin{document}

\title {Systems of parameters and the Cohen--Macaulay property}

\author {J\"urgen Herzog and  Somayeh Moradi}

\address{J\"urgen Herzog, Fachbereich Mathematik, Universit\"at Duisburg-Essen, Campus Essen, 45117
Essen, Germany} \email{juergen.herzog@uni-essen.de}

\address{Somayeh Moradi, Department of Mathematics, School of Science, Ilam University,
P.O.Box 69315-516, Ilam, Iran}
\email{so.moradi@ilam.ac.ir}

\begin{abstract}
We recall a  numerical criteria for Cohen--Macaulayness related to system of parameters, and introduce monomial ideals of K\"onig type which include the edge ideals of K\"onig graphs. We show that a monomial ideal is of K\"onig type if and only if its corresponding residue class ring admits a system of parameters whose elements are of the form $x_i-x_j$. This provides an algebraic characterization of K\"onig graphs. We use this special parameter systems for the study of the edge ideal of K\"onig graphs and the study of the order complex of a  certain family of posets. Finally,  for any simplicial complex $\Delta$ we introduce a system of parameters for $K[\Delta]$ with a universal construction principle, independent of the base field and only dependent on the faces of $\Delta$. This system of parameters is an efficient tool to test Cohen--Macaulayness of the Stanley--Reisner ring of a simplicial complex.
\end{abstract}

\keywords{System of parameters, Cohen-Macaulay, K\"{o}nig graph}

\subjclass[2010]{Primary 13H10; Secondary 05C25}

\maketitle

\setcounter{tocdepth}{1}

 \section*{Introduction}
Systems of parameters play an important role in dimension theory. As a consequence of Krull's generalized principal ideal theorem it can be seen that in a Noetherian local ring $(R,\mm)$ with $\dim R=d$, there exist elements $f_1,\ldots,f_d\in\mm$ with $\dim R/(f_1,\ldots,f_d)=0$. Such a sequence of elements of $R$ is called a system of parameters, or sop for short. A similar statement holds for standard graded $K$-algebras with  $K$ a field. In our applications  we mainly consider such algebras.

One of the central problems in Combinatorial Commutative Algebra is to show that a certain $K$-algebra attached to a combinatorial object is Cohen--Macaulay. Usually the Cohen--Macaulay property has a nice combinatorial interpretation. In the case that the defining ideal of the algebra is a monomial ideal,  Hochster's formula \cite{Ho} and its extension by Takayama \cite{Ta} are powerful tools to investigate the homological properties of the algebra. In the case that the defining ideal is a binomial prime ideal,  one may use the squarefree divisor complex \cite{BH2} or one may use Gr\"obner basis theory to reduce the problem to the case of monomial ideals.

In this paper we propose another approach which is based on the  basic fact that $R$ is Cohen--Macaulay if and only if one (equivalently all) of the sop's of $R$ form(s) a regular sequence. This approach confronts us with two problems. The first problem is to find a suitable sop, the second is to decide whether the given sop forms a regular sequence. Regarding the first problem, Stanley \cite[Proposition 4.3]{St1} finds an explicit special sop for the Stanley--Reisner ring of any balanced simplicial complex. This also  includes the order complexes.   In the cases considered here we also use special sop's.

In the first section of this paper however we  first deal with the second problem. Based on results of Serre \cite{Se}, see also \cite[Theorem 4.6.10]{BH} one has a numerical condition for when a sop is a regular sequence. Indeed, let $f_1,\ldots,f_d\in\mm$ be a  sop of $R$ and let $\overline{R}=R/(f_1,\ldots,f_d)$. Then,  denoting  by $e(M)$ the multiplicity of an $R$-module $M$, one has  $e(\overline{R})\geq e(R)$, and   if $e(\overline{R})= e(R)$ then  $f_1,\ldots,f_d$ is a regular sequence (equivalently,  $R$ is Cohen--Macaulay).  Moreover, if $(f_1,\ldots,f_d)$ is a reduction ideal of $\mm$ and $f_1,\ldots,f_d$ is a regular sequence, then $e(\overline{R})= e(R)$, see Proposition~\ref{criterion}. There is also a graded version of this criterion, see Proposition~\ref{graded}. In the case that $R$ is a standard graded $K$-algebra and the sop $f_1,\ldots,f_d$ is homogeneous with $\deg f_i =a_i$, then this sop is a regular sequence if and only if $e(\overline{R})=a_1a_2\cdots a_d e(R)$. By a lack of good references we provided  the detailed proofs of these results.

In Proposition~\ref{refinement} we give in the graded case a measure for the difference $e(\overline{R})-e(R)$. As a consequence we obtain in Corollary~\ref{surprising} the result that if the sop $f_1,\ldots,f_d$ is a superficial sequence, and $R/(f_1,\ldots,f_r)$ is Cohen--Macaulay for some $r<d$, then $R$ is Cohen--Macaulay.

In Section~\ref{sec:special} we study a class of posets and their order complexes as well as  K\"onig graphs by means of sop's. We consider a poset  $P$  which  as a set is the disjoint union of two sets $C_1$ and $C_2$, where $C_1:x_1<x_2<\cdots<x_n$ and $C_2:y_1<y_2<\cdots<y_n$  are maximal chains in $P$. For such a poset the sequence $x_1-y_1, \ldots, x_n-y_n$ is a sop of the Stanley--Reisner ring $K[\Delta(P)]$, where $\Delta(P)$ denotes the order complex of $P$. The covering relations $x_i\lessdot y_j$  in $P$ we call the diagonals of $P$. The Cohen--Macaulay property of $K[\Delta(P)]$ can be expressed in terms of the diagonals of $P$.  Indeed, in Theorem~\ref{poset} it is shown $K[\Delta(P)]$ is Cohen--Macaulay  if and only if it is pure shellable, and that this is equivalent to the condition that the diagonals of $P$ satisfy the following conditions: (i) if $x_i\lessdot y_j$ or $y_i\lessdot x_j$, then $j=i+1$, and (ii)  $\{x_i,y_{i+1}\}\notin \Delta(P)$ implies that $\{x_{i+1},y_i\}\in \Delta(P)$.  In a similar fashion it can be characterized when  $I_{\Delta(P)}$ has a linear resolution, see Proposition~\ref{linear}.

Note that $I_{\Delta(P)}$  may be viewed as the edge ideal $I(G)$ of a suitable bipartite graph $G$. So the question arises for which graphs $G$  can we find a sop $f_1,\ldots, f_d$  of $K[V(G)]/I(G)$,  where each $f_i$ is just a difference of two variables, like we have it for $K[\Delta(P)]$. The advantage of such sop's is that after reduction they preserve the monomial structure and just identify vertices.
The surprising answer to the above  question is that a graph $G$ admits such a special sop if and only if $G$ is a  K\"onig graph. In fact, this is a corollary of a more general theorem. Let $I\subset S$ be a monomial ideal in the polynomial ring  $S=K[x_1,\ldots,x_n]$ over the field $K$ in $n$ variables. We denote by $\mgrade(I)$ the maximal length of a regular sequence of monomials in $I$, and call this number the {\em monomial grade} of $I$. One has  $\mgrade(I)\leq \grade(I)=\height(I)$. We call $I$ a {\em monomial ideal of K\"onig type} if $I\neq 0$ and $\mgrade(I)=\height(I)$. The naming is justified by the fact that if $I=I(G)$ for some graph $G$, then $\height(I)=\tau(G)$ and $\mgrade(I)=\nu(G)$, so that the edge ideal of a graph $G$ is a monomial ideal of K\"onig type  if and only if $G$ is a  K\"onig graph. Now our Theorem~\ref{algebraic} says that a monomial ideal $I\subset S=K[x_1,\ldots,x_n]$  is a monomial ideal of K\"onig type if and only if $S/I$ admits a sop $f_1,\ldots,f_d$,  where each $f_k$ is of the form $x_i-x_j$  for suitable $i$ and $j$.

Applied to graphs this result reads a follows: let  $G$ be a graph without isolated vertices, $S=K[V(G)]$ and for any edge $e=\{x,y\}\in E(G)$, let $f_e=x-y$ be an element in $S$. Then $G$ is a K\"{o}nig graph if and only if there exists a subset $\{e_1,\ldots,e_d\}$ of edges of $G$ such that $f_{e_1},\ldots,f_{e_d}$ is a sop for $R=S/I(G)$. This sop has the nice property that $\reg(R/(f_{e_1},\ldots,f_{e_d})R)\leq \reg(R)$, as shown in Theorem~\ref{reg}.

For a graph $G$, we denote by  $\mi(G)$  the number of maximal independent sets of $G$. It is an important problem in graph theory to give upper bounds for $\mi(G)$. For a K\"onig graph it was shown in  \cite[Corollary 3.4]{HTT} that $2^{\nu(G)}$ is an upper bound for $\mi(G)$ where  $\nu(G)$ denotes the maximum size of matchings of $G$, and  in \cite[Theorem 1]{A} it was proved that $\mi(G)\leq M(G)+1$, where $M(G)$ is the number of induced matchings in $G$. By using our special sop for unmixed K\"onig graphs  we give  a stronger bound for $\mi(G)$ and at the same time provide a combinatorial criterion for  the Cohen--Macaulay property for unmixed K\"onig graphs. A different combinatorial characterization of Cohen-Macaulay K\"{o}nig graphs is known from  \cite[Proposition 28]{CCR}. Our result (Theorem~\ref{im}) is a follows: Let $G$ be a K\"{o}nig graph and $\{e_1,\ldots,e_m\}$ be a maximal matching of $G$ with $\nu(G)=m$,  and let $k$ be the number of induced matchings of $G$ contained in $\{e_1,\ldots,e_m\}$. Then  $\mi(G)\leq k+1$  and equality holds if and only if $G$ is a Cohen-Macaulay graph.

In the last section we introduce a sop for  the  Stanley--Reisner ring of any simplicial complex $\Delta$. We call it  the universal sop  of $K[\Delta]$ because it is built in a uniform way for all simplicial complexes,  and its construction does not depend on the base field $K$. The price we have to pay for this, is that this is not a sop of linear forms, instead it is defined as follows:   $p_i(\Delta)=\sum_{F\in \Delta\atop |F|=i}\prod_{j\in F}x_j$ for $i=1,\ldots,\dim \Delta +1$,  see Theorem~\ref{standard}. By using  Proposition~\ref{graded} we obtain a Cohen--Macaulay criterion for $K[\Delta]$ in terms of this sop. This  turns out to be a useful computational tool to check Cohen-Macaulayness, as we demonstrate at the example of  a chessboard complex.

\section{Criteria of Cohen--Macaulayness in  terms of systems of parameters}
\label{sec:crit}

In this section we collect some results on  sop's which all are based on results of Serre (see \cite[Theorem 4.6.10]{BH}) and which in terms of multiplicities allow   to check whether a ring or a module is Cohen--Macaulay.  One of the first efficient applications of these criteria was given by the first author of this paper  in order to study the conormal module and the module of differentials   of a $K$-algebra,  see \cite{JH}.

\begin{Proposition}
\label{criterion}
Let $R$ be a Noetherian local ring (or a standard graded $K$-algebra) with (graded) maximal  ideal $\mm$, and let $I\subset R$ be an ideal generated by a (homogeneous) sop of $R$. Then
\begin{enumerate}
\item[(a)] $e(R/I)=\length (R/I)\geq e(R)$.
\item[(b)] If $e(R/I)= e(R)$, then $R$ is Cohen-Macaulay.
\item[(c)] If $I$ is a reduction ideal of $\mm$ and $R$ is Cohen-Macaulay, then $e(R/I)= e(R)$.
\end{enumerate}

If {\em (b)} holds, then the sop which generates  $I$ is a regular sequence. In particular,  $r(R)=r(R/I)$,  and so $R$ is Gorenstein if and only if $R/I$ is Gorenstein. (Here we denote by $r(M)$ the (Cohen--Macaulay) type of a Cohen--Macaulay module $M$).
\end{Proposition}

\begin{proof}
For the proof we recall a few facts: Let $M\neq 0$ be a finitely generated $R$-module of dimension $d$ and $I\subseteq \mm$ be an ideal with $\dim M/IM=0$. Then
\[
e(I,M)=\lim_{k\to \infty}(d!/k^d)\length(M/I^{k+1}M)
\]
is called the multiplicity of $M$ with respect to $I$. The multiplicity of $M$, denoted $e(M)$, is the multiplicity of $M$ with respect to $\mm$.

Obviously, if $I\subseteq J\subseteq \mm$, then
$
e(I,M)\geq e(J,M).
$ In particular,
\begin{eqnarray}
\label{IJ}
e(I,M)\geq e(M).
\end{eqnarray}
On the other hand, if $I$ is a reduction ideal of $\mm$ with respect to $M$, that is, if $I\mm^k M=\mm^{k+1}M$ for some $k$, then equality holds in (\ref{IJ}), see \cite[Lemma 4.5.5]{BH}.

We also need the following result (\cite[Corollary 4.6.11]{BH} or \cite{JH} where it first appeared:  let $I$ be generated by a sop and assume that $M$ has positive rank. Then
\begin{enumerate}
\item[(i)] $\length(M/IM)\geq e(I,M)\rank M$.
\item[(ii)] $M$ is Cohen-Macaulay  if and only if  $\length(M/IM)= e(I,M)\rank M$.
\end{enumerate}

Now we apply these results to the case that $M=R$. We first notice that $e(R/I)=\length(R/I)$, since $\dim R/I=0$. Next (\ref{IJ}) and (i) imply
\begin{eqnarray}
\label{inbetween}
\length(R/I)\geq e(I,R)\geq e(R).
\end{eqnarray}
This proves (a). If $\length(R/I)= e(R)$, then (\ref{inbetween}) implies $\length(R/I)= e(I,R)$, and then (ii) yields (b). Finally, if $I$ is a reduction ideal of $\mm$, then $e(I,R)=e(R)$, and (\ref{inbetween}) together with (ii) implies (c).

If $R$ is Cohen--Macaulay, then each sop is a regular sequence.
\end{proof}

Now we turn to a  graded version of Proposition~\ref{criterion}.

\begin{Proposition}
\label{graded}
Let $R$ be a standard graded $K$-algebra with graded maximal ideal $\mm$, and  let $I$ be generated by  the homogeneous sop $f_1,\ldots,f_d$ with $\deg f_i=a_i$ for $i=1,\ldots,d$. Then
\begin{enumerate}
\item[(a)] $e(R/I)=\length (R/I)\geq a_1a_2\cdots a_d e(R).$

\item[(b)] $R$ is Cohen--Macaulay if and only if $e(R/I)= a_1a_2\cdots a_d e(R).$
\end{enumerate}
If the equivalent conditions given in {\em (b)}  hold, then $f_1,\ldots,f_d$  is a regular sequence. In particular $r(R)=r(R/I)$,  and $R$ is Gorenstein if and only if $R/I$ is Gorenstein.
\end{Proposition}

\begin{proof}
 Let $a=a_1a_2\cdots a_m$ and set $b_i=a/a_i$ for $i=1,\ldots,d$. Then
\[
e(f_1^{b_1},\ldots,f_d^{b_d},R)=b_1b_2\cdots b_de(f_1,\ldots,f_d,R),
\]
see \cite[Proposition 11.2.9]{SH}. Moreover, $\deg f_i^{b_i}= a$ for $i=1,\ldots,d$. Therefore,
\[
(f_1^{b_1},\ldots,f_d^{b_d}) \subseteq \mm^a.
\]
Thus, since $e(\mm^a,R)=a^de(R)$  (\cite[Proposition 11.2.9]{SH}), we obtain that
\begin{eqnarray}
\label{prod}
b_1b_2\cdots b_de(f_1,\ldots,f_d,R)=e(f_1^{b_1},\ldots,f_d^{b_d},R)\geq e(\mm^a,R)=a^de(R),
\end{eqnarray}
which together with (i) in the proof of Proposition~\ref{criterion} imply the inequality  in (a).

(b) Assuming that $\length (R/I)= a_1a_2\cdots a_d e(R)$, we obtain together with (\ref{prod}) that
\[
 a_1a_2\cdots a_d e(R)=\length (R/I)\geq e(f_1,\ldots,f_d,R)\geq a_1a_2\cdots a_d e(R),
 \]
and hence $\length (R/I)= e(f_1,\ldots,f_d,R)$. Thus (ii) in the proof of Proposition~\ref{criterion}  implies that $R$ is Cohen--Macaulay. Conversely, suppose that $R$ is Cohen--Macaulay. Then $f_1,\ldots,f_d$ is a regular sequence. Let $\Hilb_R(t)=Q_R(t)/(1-t)^d$ be the Hilbert series of $R$. Then
\[
Q_{R/I}(t)=\Hilb_{R/I}(t)=\Hilb_R(t)\prod_{i=1}^d(1-t^{a_i})=Q_R(t)(\prod_{i=1}^d(\sum_{j=0}^{a_{i-1}}t^j).
\]
It follows that $$\length(R/I)=Q_{R/I}(1)=Q_R(1)a_1a_2\cdots a_d=e(R)a_1a_2\cdots a_d.$$
\end{proof}

The next result is a certain refinement of the statements given in Proposition~\ref{graded}.
We first recall the following fact (see for example (\cite[Proposition A.4.]{BH}): Let $(R,\mm)$ be a Noetherian local ring and $f_1,\ldots,f_m$ a sequence of elements in $\mm$. Then
\begin{enumerate}
\item[($\alpha$)] $\dim R\geq \dim R/(f_1,\ldots,f_m)\geq \dim R-m$,  and
\item[($\beta$)] $\dim R/(f_1,\ldots,f_m)= \dim R-m$ if and only $f_1,\ldots, f_m$ can be completed to a sop of $R$.
\end{enumerate}
A similar statement holds for graded $K$-algebras.

\begin{Proposition}
\label{refinement}
With the assumptions and notation of Proposition~\ref{graded}  let
\[
U_i=\Ker(R/(f_1,\ldots,f_{i-1})\xrightarrow{f_i}
R/(f_1,\ldots,f_{i-1})).\]
Then
\begin{enumerate}
\item[(a)] $\dim U_i\leq d-i$ for all $i$.
\item[(b)] Set $\dim U_i=-1$ if $U_i=0$. Then
\[
e(R/I)=a_1\cdots a_de(R)+\sum_{i=1 \atop  \dim U_i=d-i }^da_{i+1}\cdots a_de(U_i).
\]
\end{enumerate}
In particular,  if $\deg f_i=1$ for all $i$, then
\[
e(R/I)=e(R)+\sum_{i=1 \atop  \dim U_i=d-i }^de(U_i).
\]
\end{Proposition}

\begin{proof}
(a) $U_i$ is a submodule of $R/(f_1,\ldots,f_{i-1})$ with $f_iU_i=0$. Thus $U_i$ is a $R/(f_1,\ldots,f_{i})$-module  and hence $\dim U_i\leq \dim R/(f_1,\ldots,f_{i}) =d-i$, where the equation follows from ($\beta$).

(b) From the following exact sequence
\begin{eqnarray*}
 0\to U_i\to (R/(f_1,\ldots,f_{i-1}))(-a_i)\xrightarrow{f_i}
R/(f_1,\ldots,f_{i-1})\to R/(f_1,\ldots,f_{i})\to 0
\end{eqnarray*}
we deduce the equality
\begin{eqnarray}
\label{hilbert}
\Hilb_{R/(f_1,\ldots,f_{i})}(t)=(1-t^{a_i})\Hilb_{R/(f_1,\ldots,f_{i-1})}(t)+\Hilb_{U_i}(t).
\end{eqnarray}
We have $\Hilb_{R/(f_1,\ldots,f_{i})}(t)=Q_i(t)/(1-t^{d-i})$ with $Q_i(1)=e(R/(f_1,\ldots,f_{i}))$, similarly,
$\Hilb_{R/(f_1,\ldots,f_{i-1})}(t)=Q_{i-1}(t)/(1-t^{d-(i-1)})$ with $Q_{i-1}(1)=e(R/(f_1,\ldots,f_{i-1}))$ and $\Hilb_{U_i}(t)  =P_i(t)/(1-t^{\delta_i})$ with $P_i(1)=e(U_i)$ and $\delta_i\leq d-i$.

Thus (\ref{hilbert}) implies that
\begin{eqnarray*}
Q_i(t)/(1-t^{d-i})=(1-t^{a_i})(Q_{i-1}(t)/(1-t^{d-(i-1)}))+
P_i(t)/(1-t^{\delta_i}),
\end{eqnarray*}
from which we deduce that
\[
Q_i(t)=Q_{i-1}(t)(\sum_{j=0}^{a_i-1}t^j)+(1-t)^{d-i-\delta_i}P_i(t).
\]
Substituting $t$ by $1$,  we get
\[
e(R/(f_1,\ldots,f_i))=
\begin{cases}
a_ie(R/(f_1,\ldots,f_{i-1})), \text{\;\; \; \;\; \;\;\;\     if $\dim U_i<d-i$},\\
a_ie(R/(f_1,\ldots,f_{i-1}))+e(U_i), \text{ if $\dim U_i=d-i$}.
\end{cases}
\]
These formulas together with  induction on $i$ complete the proof.
\end{proof}

Proposition~\ref{refinement} together with Proposition~\ref{graded} has the following a surprising consequence

\begin{Corollary}
\label{surprising}
With the assumptions and notation of Proposition \ref{refinement} let $r<d$ be an integer with the property that $\dim U_i<d-i$ for $i=1,\ldots,r$  and that $R/(f_1,\ldots,f_r)$ is Cohen--Macaulay. Then $R$ is Cohen-Macaulay and $U_i=0$ for  all $i$. In particular, if $f_1,\dots,f_d$ is a superficial sequence and $R/(f_1,\ldots,f_i)$ is Cohen-Macaulay for some $i<d$, then $R$ is Cohen--Macaulay.
\end{Corollary}

\begin{proof}
Proposition~\ref{refinement} implies that
$e(R/(f_1,\ldots,f_r))=a_1\cdots a_re(R)$. Since by assumption $R/(f_1,\ldots,f_r)$ is Cohen--Macaulay,  it follows that $f_{r+1},\cdots, f_d$ is a regular $R/(f_1,\ldots,f_r)$-sequence. Hence,  $e(R/(f_1,\ldots,f_d))=a_{r+1}\cdots a_de(R/(f_1,\ldots,f_r))$, and we deduce that  $e(R/(f_1,\ldots,f_d))=a_1\cdots a_d e(R)$. Thus the desired result follows from Proposition~\ref{graded}.
\end{proof}

\medskip
We close this section with a remark  and a question. Let $S=K[x_1,\ldots,x_n]$ be the polynomial ring over the field $K$, $R=S/I$ with $I\subset (x_1,\ldots,x_n)^2$ a graded ideal. Let $f_1,\ldots,f_d$ be linear forms of $S$ which form a sop for $R$. Let $\overline{S}=S/(f_1,\ldots,f_d)$. Then $\overline{S}$ is isomorphic to a polynomial ring in $n-d$ variables, and  $R/(f_1,\ldots,f_d)=\bar{S}/\overline{I}$, where $\overline{I}=I\overline{S}$.

\begin{Remark}
\label{question}
With the notation introduced we have $\projdim_{\overline{S}} \overline{I} \leq \projdim_S I $ and $\mu(\overline{I})\leq \mu(I)$. Equality holds in both inequalities,  if $S/I$ is Cohen--Macaulay.
\end{Remark}

\begin{proof}
By the Auslander-Buchsbaum formula,
\[
\projdim_S S/I=n-\depth S/I\geq n-\dim S/I=n-d =\projdim_{\overline{S}}\overline{S}/\overline{I}.
\]
The last equation holds since $\dim \overline{S}/\overline{I}=0$. This implies the first assertion. It is obvious that $\mu(I)\geq \mu(I\overline{S})$. Finally, if $S/I$ is Cohen--Macaulay, then $I$ is generated by a regular sequence  and the desired equalities hold.
\end{proof}

Remark~\ref{question} implies in particular that if  $\mu(\overline{I})< \mu(I)$, then $R$ cannot be Cohen--Macaulay. On the other hand,  $\mu(\overline{I})=\mu(I)$, does not necessarily imply that  $R$  is Cohen--Macaulay. For example for the cycle graph $C_6:x_1,\ldots,x_6$,  the sequence $x_1-x_2,x_3-x_4,x_5-x_6$ is a  sop for $R=K[x_1,\ldots,x_6]/I(C_6)$ and $\overline{I(C_6)}=(x_1^2,x_3^2,x_5^2,x_1x_3,x_3x_5,x_1x_5)$. Then $\mu(\overline{I(C_6)})=\mu(I(C_6))=6$, while $R$ is not Cohen-Macaulay.

In view of these inequalities  one is tempted to ask whether under the assumptions of Remark~\ref{question} we have $\reg(\overline{I})\leq \reg(I)$. In the next section we show that this inequality for the regularity indeed holds for  the edge ideal of K\"onig graphs and suitable  natural sop's.

\section{Special systems of parameters applied to order complexes and K\"onig graphs}
\label{sec:special}

In this section we define monomial ideals of  K\"onig type which include edge ideals of  K\"onig graphs and give a characterization for these ideals in terms some sop's for their quotient rings. Also we apply Proposition \ref{criterion} to the Stanley-Reisner ring of two families of simplicial complexes, namely the order complex of a finite poset and the independence complex of a K\"{o}nig graph and give combinatorial descriptions for the Cohen-Macaulay property of these rings.

For a poset $P$, a nonempty subposet $C$ of $P$ which is totally
ordered is called a \emph{chain} in $P$. The \emph{order complex} of $P$ denoted by $\Delta(P)$ is the simplicial complex whose faces are the chains in $P$.
The \emph{length} of a chain $C$ in $P$ is defined to be $|C|-1$. The \emph{height} of an element $x$ in $P$ is defined to be the maximal
length of a chain descending from $x$. For elements $x$ and $y$ in a poset $P$, it is said that $y$ \emph{covers} $x$, denoted $x\lessdot y$, if $x<y$ and there exists
no $z\in P$ such that $x<z<y$.
Also for a monomial ideal $I$, the cardinality of any minimal generating set of monomials of $I$ is denoted by $\mu(I)$.

\begin{Theorem}\label{poset}
Let $P$ be a poset which  as a set is the disjoint union of two sets $C_1$ and $C_2$, where $C_1:x_1<x_2<\cdots<x_n$ and $C_2:y_1<y_2<\cdots<y_n$  are maximal chains in $P$ and let $\Delta=\Delta(P)$.
Then the following conditions are equivalent:
\begin{enumerate}
\item[(a)]  $\Delta$ is Cohen-Macaulay.
\item[(b)]  $\Delta$ is pure shellable.
\item[(c)] $\Delta$ satisfies the following conditions:
\begin{enumerate}
             \item [(1)] If $x_i\lessdot y_j$ or $y_i\lessdot x_j$, then $j=i+1$, and
             \item [(2)] $\{x_i,y_{i+1}\}\notin \Delta$ implies that $\{x_{i+1},y_i\}\in \Delta$.
           \end{enumerate}
\end{enumerate}
\end{Theorem}

\begin{proof}
(b)\implies (a): By \cite[Theorem 8.2.6]{HH} the assertion holds.

(a) \implies (c):  Suppose that $\Delta$ is Cohen-Macaulay. Then  $\Delta$ is pure.
Suppose $x_i\lessdot y_j$ for some $i$ and $j$. If $j\leq i$, then the chain $x_1<\cdots<x_i<y_j<y_{j+1}<\cdots<y_n$ is a chain of cardinality at least $n+1$, which is included in some maximal chain of $P$. But by purity of $\Delta$ any maximal chain should have cardinality $|C_1|=n$, which gives a contradiction. Thus $i<j$.
Similarly if $y_i\lessdot x_j$ for some $i$ and $j$, then $i<j$.
Now assume that $x_i\lessdot y_j$ and by contradiction let $j\neq i+1$. Since $j>i$, one should have $j>i+1$. Then $x_1<\cdots<x_i<y_j<y_{j+1}<\cdots<y_n$ is a maximal chain of cardinality at most $n-1$ in $P$, which is again a contradiction to purity of $\Delta$. So $j=i+1$. The argument for the case  $y_i\lessdot x_j$ is similar.

To prove $(2)$, first we show that the sequence $x_1-y_1,x_2-y_2,\ldots,x_n-y_n$ is a sop for the ring $R=S/I_{\Delta}$, where $S=K[x_1,\ldots,x_n,y_1,\ldots,y_n]$.
Indeed by $(1)$, $x_iy_i\in I_{\Delta}$ for any $1\leq i\leq n$. So we have $x_i^2,y_i^2\in (I_{\Delta},x_1-y_1,x_2-y_2,\ldots,x_n-y_n)$ for all $i$ and then $\dim R/(x_1-y_1,x_2-y_2,\ldots,x_n-y_n)R=0$. Also $\dim R=\dim \Delta+1=n$. Now,  by contradiction suppose that for some $i$, $\{x_i,y_{i+1}\}\notin \Delta$ and $\{x_{i+1},y_i\}\notin \Delta$. This means that $x_iy_{i+1}$ and $x_{i+1}y_i$ belong to the set of minimal generators  $\mathcal{G}(I_{\Delta})$ of $I_{\Delta}$. One has $R/(x_1-y_1,x_2-y_2,\ldots,x_n-y_n)R\cong K[x_1,\ldots,x_n]/I'$, where $I'=(x_ix_j:\ \textrm{$x_i$ and $y_j$ are non-comparable in $P$})$.
Since $R$ is Cohen-Macaulay, by Remark \ref{question}, one should have $\mu(I_{\Delta})=\mu(I')$. But
since $\mathcal{G}(I_{\Delta})=\{x_iy_j:\ \textrm{$x_i$ and $y_j$ are non-comparable in $P$}\}$, and $x_iy_{i+1},x_{i+1}y_i\in \mathcal{G}(I_{\Delta})$ correspond to just one element in $\mathcal{G}(I')$ that is $x_ix_{i+1}$, we have  $\mu(I')<\mu(I_{\Delta})$, a contradiction. Thus  $\{x_i,y_{i+1}\}\in \Delta$ or $\{x_{i+1},y_i\}\in \Delta$.

(c)\implies (b): Let $P$ be a poset satisfying the assumptions of (c).  Let $F=\{z_1<z_2<\cdots<z_k\}$ be an arbitrary facet  of $\Delta$. First note that $z_1\lessdot z_2\lessdot \cdots\lessdot z_k$. Also assumption $(1)$ of (c) implies that $\{x_i,y_i\}\notin \Delta$ for any $1\leq i\leq n$, and then $|\{x_i,y_i\}\cap F|\leq 1$. We claim that for each facet $F$ of $\Delta$,
 $|\{x_i,y_i\}\cap F|=1$ for any  $1\leq i\leq n$.  One can easily see that for $i=1$ the claim holds true. Indeed if $z_1\in C_1$, then $z_1=x_1$, because otherwise $F \subsetneq F\cup\{x_1\}\in \Delta$, a contradiction. Similarly if $z_1\in C_2$, then $z_1=y_1$. So $|\{x_1,y_1\}\cap F|=1$.   Assume inductively that for any $i=1,\ldots,m-1$, $|\{x_i,y_i\}\cap F|=1$. We show that $|\{x_m,y_m\}\cap F|=1$. We have $x_{m-1}\in F$ or
$y_{m-1}\in F$. Without loss of generality suppose $x_{m-1}\in F$. If $x_m\in F$, we are done. So assume that $x_m\notin F$. Note that $x_{m-1}\neq z_k$, since otherwise $F \subsetneq  F\cup\{x_m\}\in \Delta$, a contradiction. So there exists
$z_t\in F$ such that $x_{m-1}\lessdot z_t$. If $z_t=x_j$ for some $j$, then $j>m$
and hence $x_{m-1}< x_m< z_t$, which contradicts to $x_{m-1}\lessdot z_t$.
Thus $z_t=y_j$ for some $j$ and by $(1)$, $j=m$. Thus $z_t=y_m\in F$. So $|\{x_m,y_m\}\cap F|=1$.
Therefore any facet $F$ of  $\Delta$ has cardinality $n$ such that for any  $1\leq i\leq n$, either $x_i\in F$ or $y_i\in F$. Thus  $\Delta$ is pure.

Let $\mathcal{F}(\Delta)$ denotes the set of facets of $\Delta$. To prove the shellability, consider the ordering on $\mathcal{F}(\Delta)$ as follows. For the facets $F_i$ and $F_j$ of $\Delta$, we set $F_i\prec F_j$ if there exists $1\leq t\leq n$ such that $y_t\in F_j\setminus F_i$ and for any $d<t$, the elements of height $d$ in $F_i$ and $F_j$ are the same. Now, let $F_i\prec F_j$ and $1\leq t\leq n$ be such that $y_t\in F_j\setminus F_i$ and for any $d<t$, the elements of height $d$ in $F_i$ and $F_j$ are the same.
If $\{x_t,y_{t+1}\}\in \Delta$, then we set $F_k=(F_j\setminus \{y_t\})\cup\{x_t\}$. Then $F_k\in \mathcal{F}(\Delta)$, $F_k\prec F_j$ and $F_j\setminus F_k=\{y_t\}$. So we may assume that
$\{x_t,y_{t+1}\}\notin \Delta$. Thus by assumption, $\{x_{t+1},y_t\}\in \Delta$. Let $r$ be the greatest integer with $y_t,y_{t+1},\ldots,y_r\in F_j\setminus F_i$
and $\{x_{s+1},y_s\}\in \Delta$ for any $1\leq s\leq r$. Note that $r$ is well-defined, since $y_t\in F_j\setminus F_i$ and $\{x_{t+1},y_t\}\in \Delta$. Two cases may happen:

$(1)$ Suppose that $y_{r+1}\in F_j\setminus F_i$. Then by the maximality of $r$, $\{x_{r+2},y_{r+1}\}\notin \Delta$. So by assumption $\{x_{r+1},y_{r+2}\}\in \Delta$.
Then  $F_k=(F_j\setminus \{y_{r+1}\})\cup\{x_{r+1}\}\in \mathcal{F}(\Delta)$, $F_k\prec F_j$ and $F_j\setminus F_k=\{y_{r+1}\}$.

$(2)$  Suppose that $y_{r+1}\notin F_j\setminus F_i$. Then either $y_{r+1}\notin F_j$ or $y_{r+1}\in F_i\cap F_j$. In both cases, we set $F_k=(F_j\setminus \{y_r\})\cup\{x_r\}\in \mathcal{F}(\Delta)$ and $F_k$
is the facet which fulfils the desired condition for shellability.
\end{proof}

%The  covering relations $x_i\lessdot y_j$ between the chains $C_1$ and $C_2$ are sometimes called diagonals of $P$.
The next result shows that the linear resolution  property for  $I_{\Delta (P)}$  can be again expressed in terms of conditions on chains of the form $x_i<y_j$ in $P$.
For a graph $G$ by $V(G)$ and $E(G)$ we mean the vertex set and the edge set of $G$, respectively.
Also for a subset $S\subseteq V(G)$, the induced subgraph of $G$ on the set $S$ is denoted by $G_S$.

\begin{Proposition}
\label{linear}
Let $P$ be a poset with the assumptions of Theorem \ref{poset} and $\Delta=\Delta(P)$. Then $I_{\Delta}$ has a linear resolution if and only if
whenever  $\{x_i,y_j\},\{x_r,y_s\}\in \Delta$, then  $\{x_i,y_s\}\in \Delta$ or $\{x_r,y_j\}\in \Delta$.
\end{Proposition}

\begin{proof}
We have $I_{\Delta }=I(G)$, where $G$ is a bipartite graph with bipartition $X=\{x_1,x_2,\ldots,x_n\}$, $Y=\{y_1,y_2,\ldots,y_n\}$ and the edge set $E(G)=\{\{x_i,y_j\}:\ \{x_i,y_j\}\notin \Delta\}$.
Thus by \cite[Theorem 1]{Fro},  $I_{\Delta }$ has a linear resolution if and only if $G^c$ is a chordal graph.
Any cycle $C$ of length $m\geq 5$, has a chord in $G^c$, because $C$ has at least $3$ vertices from $X$ or from $Y$, and $G^c_X$ and $G^c_Y$ are complete graphs.
So $G^c$ is chordal if and only if any cycle of length $4$ in $G^c$ has a chord.
Note that $C:x_i,x_r,y_s,y_j$ is a cycle in $G^c$ if and only if  $\{x_i,y_j\},\{x_r,y_s\}\in \Delta$
and it has a chord if and only if $\{x_i,y_s\}\in \Delta$ or $\{x_r,y_j\}\in \Delta$.
Thus  $G^c$ is a chordal graph if and only if whenever $\{x_i,y_j\},\{x_r,y_s\}\in \Delta$, then $\{x_i,y_s\}\in \Delta$ or $\{x_r,y_j\}\in \Delta$.
%Now, we show that $G^c$ is a chordal graph if whenever $\{x_i,y_j\},\{x_r,y_s\}$ are diagonals of $P$, then $\{x_i,y_s\}\in \Delta(P)$ or $\{x_r,y_j\}\in \Delta(P)$.
%Let $C:x_i,x_r,y_s,y_j$ be a cycle in $G^c$ and without loss of generality assume that $i\leq j$. Then $x_i\lessdot y_{j'}$ for some $j'\leq j$.
\end{proof}

\medskip

For a graph $G$, let $\tau(G)$ be the minimum cardinality of a vertex cover of $G$ and $\nu(G)$ denotes the maximum cardinality of a matching of $G$. One can see that for any graph $G$, $\tau(G)\geq \nu(G)$.
Recall that a graph $G$ is called a {\em  K\"{o}nig graph}, when this inequality becomes an equality.
%We close this section with a version of Theorem~\ref{konigsop} which applies to all monomial ideals.
Let $I\subset S$ be a monomial ideal in the polynomial ring  $S=K[x_1,\ldots,x_n]$ over the field $K$ in $n$ variables. We denote by $\mgrade(I)$ the maximal length of a regular sequence of monomials in $I$, and call this number the {\em monomial grade} of $I$. One has  $\mgrade(I)\leq \grade(I)=\height(I)$. We call $I$ a {\em monomial ideal of K\"onig type} if $I\neq 0$ and $\mgrade(I)=\height(I)$. The naming is justified by the fact that if $I=I(G)$ for some graph $G$, then $\height(I)=\tau(G)$ and $\mgrade(I)=\nu(G)$, so that the edge ideal of K\"onig graphs are the monomial ideals of K\"onig type among edge ideals.

The following theorem characterizes monomial ideals of K\"onig type in terms of existence of some forms of sop's for their quotient rings.

\begin{Theorem}
\label{algebraic}
Let $I\subset S=K[x_1,\ldots,x_n]$ be a monomial ideal. Then the following conditions are equivalent:
\begin{enumerate}
\item[(a)] $I$ is a  monomial ideal of K\"onig type.
\item[(b)] $S/I$ admits a sop $f_1,\ldots,f_d$,  where each $f_k$ is of the form $x_i-x_j$  for suitable $i$ and $j$.
\end{enumerate}
\end{Theorem}

\begin{proof}
(a) \implies (b):  Let $h=\height (I)$. Since $I$ is a monomial ideal of  K\"onig type, there exists a regular sequence of monomials $u_1,\ldots,u_h$ with $u_i\in I$ for $i=1,\ldots,h$. We may assume that each $u_i$ belongs to the unique minimal set of monomial generators  $\mathcal{G}(I)$  of $I$. Indeed, if $u_i=wv$ with $v\in \mathcal{G}(I)$, then we may replace $u_i$ by $v$ in the above regular sequence.

For a monomial $u\in S$ we set $\supp(u)=\{i\:  x_i|u\}$. We may assume that $\Union_{u\in \mathcal{G}(I)}\supp(u)=[n]$. Indeed, suppose this is not the case. Then for simplicity we may assume that $\Union_{u\in \mathcal{G}(I)}\supp(u)=\{1,\ldots,r\}$. Note that $r\geq 1$, since $I\neq 0$. Let $R=K[x_1,\ldots,x_r]/(u\:\; u\in \mathcal{G}(I))$. Then $S/I=R[x_{r+1},\ldots, x_n]$, and $x_1-x_{r+1},\ldots, x_1-x_n$ is part of a sop of $S/I$  and $(S/I)/(x_1-x_{r+1},\ldots, x_1-x_n)\iso R$. Thus if $R$ has the desired sop, then so does $S/I$.

We proceed by induction on $d=\dim S/I$. If $d=0$, then there is nothing to show. Suppose now that $d>0$. Then at least one  $u_i$ is not a  pure power. Assume this is not the case. After a relabeling of the variables we may then assume that $u_i=x_i^{a_i}$ with $a_i>1$ for $i=1,\ldots, h$. Since $d>0$ and $\Union_{u\in \mathcal{G}(I)}\supp(u)=[n]$, there must exist $u\in \mathcal{G}(I)$, with $\supp(u)\not\in \{1,\ldots,h\}$. Then $\height(I)\geq \height (u_1,\ldots,u_h,u)\geq h+1$, a contradiction.

We may assume that $u_1$ is not a pure power, say $u_1=x_1^{a_1}x_2^{a_2}\cdots x_n^{a_n}$ with $a_1, a_2>0$. Let $f_1=x_1-x_2$. We claim that $f_1$ is a part of a sop. In other words, $f_1$ is not contained in any minimal prime ideal $P$ of $I$ with $\height(P) =\height(I)=h$. Indeed, let $P$ be a minimal prime ideal containing $f_1$. Since $I$ is a monomial ideal, $P$ is a monomial prime ideal. Therefore, if $f_1\in P$, then $x_1,x_2\in P$.  For each $i$ there exists $j_i$ such that $x_{j_i}$ divides $u_i$ and $x_{j_i}\in P$. Thus, $Q=(x_1,x_2, x_{j_2},\cdots, x_{j_d}) \subset P$. Since $u_1,\ldots,u_h$ is a regular sequence, the supports of the $u_i$ are pairwise disjoint. This implies that  the variables generating $Q$ are  pairwise distinct. It  follows that $\height(P)\geq h+1$, and $f_1$ is a part of a sop.

Identifying $x_1$ with $x_2$, we see that $(S/I)/(f_1)\iso K[x_2,\ldots,x_n]/\overline{I}$,  where $\overline{I}= (x_2^{a_1+a_2}\cdots x_n^{a_n}, u_2,\ldots,u_h, \cdots)$. This shows that $\mgrade(\overline{I})\geq h$. Since $f_1$ is a parameter element of $S/I$ it follows that $\height(\overline{I})=\height (I)=h$ and  since in general $\mgrade(\overline{I})\leq \height(\overline{I})$, we must have $\mgrade(\overline{I})= \height(\overline{I})$. This means that $\overline{I}$ is a again monomial ideal of K\"onig type. Since $\dim((S/I)/(f_1))= d-1$, we may apply our induction hypothesis, and find a sop $f_2,\ldots, f_d$ of $(S/I)/(f_1)$ as required in (b), Then $f_1, f_2,\ldots,f_d$ is the desired  sop for $S/I$,

(b)\implies (a): Let $R=S/I$ and  $\overline{R}=R/(f_1,\ldots,f_d)R$.  Then $\overline{R}\iso \overline{S}/\overline{I}$, where $\overline{I}$ is a monomial ideal, because reduction modulo  $f_1,\ldots,f_d$ simply identifies variables. For simplicity we may assume that $\overline{S}=K[x_1,\ldots,x_h]$. Since $\dim(\overline{R})=0$ and since $\overline{I}$ is a monomial ideal, it follows that $\mathcal{G}(\overline{I})$ contains a  pure power $x_i^{a_i}$ of each the variables $x_1,\ldots,x_h$. Let $u_1,\ldots,u_h$ be generators of $I$ with the property that $u_i$ specializes to $x_i^{a_i}$ under the reduction modulo $f_1,\ldots,f_d$. Suppose  $u_i$ and $u_j$ have a common factor for $i\neq j$. Then this is also the case for $x_i^{a_i}$ and $x_j^{a_j}$, a contradiction. Therefore,  $u_1,\ldots,u_h$ is a regular sequence, and so $\mgrade (I)=\height(I)$.
\end{proof}

Applying Theorem \ref{algebraic} to the edge ideal of K\"{o}nig graphs, we have the following algebraic characterization for a K\"{o}nig graph $G$ in terms of  special sop's for $R=S/I(G)$.
Note that by the proof of Theorem \ref{algebraic}, each elements $x_i-x_j$ of the sop for $R=S/I(G)$ corresponds to an edge $e=\{x_i,x_j\}$ of $G$.

\begin{Corollary}\label{konigsop}
Let $G$ be a graph without isolated vertices, $S=K[V(G)]$ and for any edge $e=\{x,y\}\in E(G)$, let $f_e=x-y$ be an element in $S$. Then $G$ is a K\"{o}nig graph if and only if there exists a subset $\{e_1,\ldots,e_d\}$ of edges of $G$ such that $f_{e_1},\ldots,f_{e_d}$ is a sop for $R=S/I(G)$.
\end{Corollary}

Let $R$ be a graded ring and $M$  a finitely generated graded $R$-module.
%Remind that the sequence $f_1,f_2,\ldots,f_m$ is called an $M$-superficial sequence if the $R$-module $(f_1,f_2,\ldots,f_{i-1})M:_M f_i$ has finite length for any $1\leq i\leq m$.
It is known that for an $M$-superficial sequence  $f_1,\ldots,f_m$ of linear forms $\reg(R/(f_1,\ldots,f_m)R)\leq \reg(R)$, see \cite[Proposition 20.20]{Eis}.
In view of this fact it is natural to ask the following

\begin{Question}\label{conjec1}
Let $R$ be a graded ring, $M$ be a finitely generated graded $R$-module and $f_1,\ldots,f_m$ be a sop of linear forms for $M$.
Does the inequality  $$\reg(M/(f_1,\ldots,f_m)M)\leq \reg(M)$$ hold?
\end{Question}

In the following theorem we  prove   the expected inequality for the $S$-module $R=S/I(G)$, when $G$ is a  K\"{o}nig graph and  the sop is of a natural special form.
First we recall some definitions. Two edges $e$ and $e'$ of a graph $G$ form a \emph{gap}, when no endpoint of $e$ and $e'$ are adjacent in $G$. Otherwise we say that $e$ and $e'$ are \emph{adjacent}.  Moreover, a subset $A$ of edges of $G$ is called an \emph{induced matching} if any two edges in $A$ form a gap. By $\a(G)$ we mean the maximum cardinality of an induced matching of $G$. The maximum cardinality of an independent set of $G$ is denoted by $\alpha(G)$.
For a graph $G$, with the vertex set $\{x_1,\ldots,x_n\}$ the\emph{ whiskered graph} of $G$ is a graph which is obtained by adding new vertices $\{y_1,\ldots,y_n\}$ and edges $\{\{x_i,y_i\}:\  1\leq i\leq n\}$ to $G$. This new graph is denoted by $G\cup W(G)$ and the edges $\{x_i,y_i\}$ are called \emph{whiskers}.

\begin{Theorem}\label{reg}
Let $G$ be a  graph, $S=K[V(G)]$ and $\{e_1,\ldots,e_m\}\subseteq E(G)$ such that $f_{e_1},\ldots,f_{e_m}$ is a sop for $R=S/I(G)$.
Then $$\reg(R/(f_{e_1},\ldots,f_{e_m})R)\leq \reg(R).$$
\end{Theorem}

\begin{proof}
Let $e_i=\{x_i,y_i\}$ for any $1\leq i\leq m$.
Then by the proof of Theorem \ref{algebraic},  $\bigcup_{i=1}^m e_i=V(G)$ and without loss of generality
we may assume that for some $r\leq m$, $C=\{x_1,x_2,\ldots,x_r\}$ is a minimal vertex cover of $G$ and $\{e_1,\ldots,e_r\}$ is a maximal matching of $G$.
Set $\overline{R}=R/(f_{e_1},\ldots,f_{e_m})R$.
Then $\overline{R}\cong K[x_1,x_2,\ldots,x_r]/L$ where
$$L=(x_1^2,\ldots x_r^2,x_ix_j:\ i<j\leq r,\ \textrm{$e_i$ and $e_j$ are adjacent in $G$}).$$
Thus $\reg(\overline{R})=\reg( K[x_1,\ldots,x_r]/L)$.
Since polarization does not change the regularity, one has  $\reg( K[x_1,\ldots,x_r]/L)=\reg(T/I(H'))$, where $T=K[x_1,\ldots,x_r,x'_1,\ldots,x'_r]$ for new variables $x'_1,\ldots,x'_r$ and $H'=H\cup W(H)$, where $H$ is  a graph on $\{x_1,\ldots,x_r\}$ with $$E(H)=\{\{x_i,x_j\}:\ i<j\leq r,\  \textrm{$e_i$ and $e_j$ are adjacent in $G$}\}$$ and  $\{x_i,x'_i\}$ is a whisker of $H'$ for any $1\leq i\leq r$.
Since any whiskered graph is very well-covered, by \cite[Theorem 1.3]{MMCRY}, $\reg(T/I(H'))=\a(H')$. One can easily see that
$\a(H')$ is precisely the maximum size of independent sets of vertices in  $H$. Moreover, by the definition of $H$, a set $\{x_{s_1},\ldots,x_{s_t}\}$ of vertices of $H$ is an independent set of $H$ if and only if $\{e_{s_1},\ldots,e_{s_t}\}\subseteq \{e_1,\ldots,e_m\}$ is an induced matching of $G$. Therefore $\a(H')\leq \a(G)$. Thus $\reg(\overline{R})=\reg( K[x_1,\ldots,x_r]/L)=\reg(T/I(H'))=\a(H')\leq \a(G)\leq \reg(R)$. The last inequality holds by \cite[Lemma 2.2]{K}.
\end{proof}

In general the inequality of Theorem \ref{reg} may be strict. Let $G=C_{n}$ denotes the cycle graph with $n$ vertices. Suppose $E(G)=\{e_i=\{x_i,x_{i+1}\}: 1\leq i\leq n-1\}\cup\{e_{n}=\{x_1,x_{n}\}\}$ and $n\equiv 2 \mod 4$. Let  $n=4m+2$ for some $m\geq 1$. Then $\{f_{e_{2i-1}}: \ 1\leq i\leq 2m+1\}$ is a sop for $R=S/I(G)$.
With the notation used in the proof of Theorem \ref{reg}, $H'=C_{2m+1}\cup W(C_{2m+1})$ and $\reg(\overline{R})=\a(H')=\alpha(H)=m$. Since
$\{e_{3i+1}: 0\leq i\leq m\}$ is an induced matching of $G$ of cardinality $m+1$, $\reg(R)\geq a(G)\geq m+1>m=\reg(\overline{R})$.

\vskip 3mm

For a graph $G$, let $\mi(G)$ denote the number of maximal independent sets of $G$.
After  Erd\"{o}s and Moser considered the problem of determining the largest value of $\mi(G)$
in terms of the number of vertices of $G$, investigating this number and upper bounds for it has been studied for various classes of graphs.
In \cite[Corollary 3.4]{HTT}  it was shown that for a K\"{o}nig graph $G$, $2^{\nu(G)}$ is an upper bound for $\mi(G)$.
Also in \cite[Theorem 1]{A} it was proved that $\mi(G)\leq M(G)+1$, where $M(G)$ is the number of induced matchings in $G$.
In the following for an unmixed K\"{o}nig graph $G$, we use a sop of the form $f_{e_1},\ldots,f_{e_d}$ for $S/I(G)$ and improve the upper bound for $\mi(G)$.
 Moreover, we give a combinatorial description of the Cohen-Macaulay property for unmixed K\"{o}nig graphs. A different combinatorial characterization of Cohen-Macaulay K\"{o}nig graphs is presented in \cite[Proposition 28]{CCR}.
 Recall that a graph $G$ is called \emph{unmixed} if all the minimal vertex covers (maximal independent sets) of $G$ are of the same cardinality.

\begin{Theorem}\label{im}
Let $G$ be an unmixed K\"{o}nig graph, $\{e_1,\ldots,e_m\}$ be a maximal matching of $G$ with $\tau(G)=\nu(G)=m$ and $k$ be the number of induced matchings of $G$ contained in $\{e_1,\ldots,e_m\}$. Then
\begin{enumerate}
\item[($a$)] $\mi(G)\leq k+1$,  and
\item[($b$)] $G$ is a Cohen--Macaulay graph if and only if $\mi(G)=k+1$.
\end{enumerate}
\end{Theorem}

\begin{proof}
Let $S=K[V(G)]$, $R=S/I(G)$ and $\dim(R)=d$. By the proof of Theorem~\ref{konigsop}, there exist $e_{m+1},\ldots,e_d\in E(G)$ such that $f_{e_1},\ldots,f_{e_d}$ is a sop for $R$.
We may assume that $\{x_1,x_2,\ldots,x_m\}$ is a minimal vertex cover of $G$, where $x_i\in e_i$ for all $1\leq i\leq m$.
Then $R/(f_{e_1},\ldots,f_{e_d})R\cong K[x_1,x_2,\ldots,x_m]/L$, where
$$L=(x_1^2,x_2^2,\ldots,x_m^2,x_ix_j:\ i<j\leq m,\ \textrm{$e_i$ and $e_j$ are adjacent in $G$}).$$
So $\length (R/(f_{e_1},\ldots,f_{e_d})R)=\length (K[x_1,x_2,\ldots,x_m]/L)$. Since $x_1^2,x_2^2,\ldots,x_m^2\in L$, any basis element of $K[x_1,x_2,\ldots,x_m]/L$ other than $1$, is an squarefree monomial of the form $x_{i_1}\cdots x_{i_r}$, where $\{e_{i_1},\ldots,e_{i_r}\}\subseteq \{e_1,\ldots,e_m\}$ is an induced matching of $G$.
So $\length (K[x_1,x_2,\ldots,x_m]/L)=k+1$.
Note that $R$ is the Stanley-Reisner ring of the independence complex $\Delta_G$ of $G$. Since $G$ is unmixed, $e(R)$ is the number of facets of $\Delta_G$ which is equal to $\mi(G)$.
Now, by using Proposition \ref{graded}, $$\mi(G)=e(R)\leq \length (R/(f_{e_1},\ldots,f_{e_d})R)=\length (K[x_1,x_2,\ldots,x_m]/L)=k+1$$ and equality holds if and only if $G$ is Cohen-Macaulay.
\end{proof}

Theorem \ref{im} shows that in a Cohen-Macaulay K\"{o}nig graph $G$, no matter which maximal matching $M$ of $G$ we choose, if $|M|=\nu(G)$, then the number of induced matchings of $G$ contained in
$M$ is equal to $\mi(G)$.

\section{A universal system of parameters for Stanley--Reisner rings}
\label{sec:universal}

Let $K$ be a field and $\Delta$ an arbitrary simplicial complex on $[n]$  of dimension $d-1$,  and let $S=K[x_1,\ldots, x_n]$ be the polynomial ring in $n$ variables.
We show that there exists a universal standard sop for $K[\Delta]=S/I_{\Delta}$. Using this sop we present a criterion for the Cohen-Macaulayness of $K[\Delta]$.

For $i=1,\dots,d $, we set
\[
p_i(\Delta)=\sum_{F\in \Delta\atop |F|=i}\xb_F,
\]
where $\xb_F=\prod_{i\in F}x_i$.
\begin{Theorem}
\label{standard}
The residue classes of  the elements $p_1(\Delta),\ldots,p_d(\Delta)$ in $K[\Delta]$ form a sop of $K[\Delta]$.
\end{Theorem}

\begin{proof}
We first consider the case that $\Delta$ is the $n$-simplex $\Gamma_n$. In that case $K[\Delta]=S$, and we have to show that  $\dim S/(p_1(\Delta),\ldots,p_n(\Delta))=0$. In order to simplify notation we write $p_i$ for $p_i(\Delta)$ and all $i$.

Note that
\begin{eqnarray}
\label{pi}
p_i=\sum_{|F|=i}\xb_F.
\end{eqnarray}
Let $<$ denote the reverse lexicographical order, and let $J=(p_1,\ldots,p_n)$.  We claim that $x_i^i\in\ini_<(J)$ for $i=1,\ldots,n$. Then this shows that indeed  $\dim S/J=0$.

For $i=1,\ldots,n$, let
\[
g_i=\sum_{j=1}^i(-1)^{j+1}x_i^{i-j}p_j.
\]
We will show that $x_i^i=\ini_<(g_i)$ for $i=1,\ldots,n$.

Observe first  that $g_i$ is homogeneous of degree $i$ and that  $x_i^i\in\supp(g_i)$. In order to complete the proof of the claim, we have to show that if $u$ is a monomial of degree $i$ with $u>x_i^i$,  then $u\not \in \supp(g_i)$.

Let $v$ be a monomial in  the support of $g_i$, $v=x_1^{a_1}\cdots x_n^{a_n}$.  Then $a_j\leq 1$ for all $j\neq i$. Hence if $x_j^a$ divides $u$ for $j\neq i$ and $a>1$, then $u\not \in \supp(g_i)$. Therefore, we may assume that $u=x_Gx_i^k$ with $0\leq k\leq i$, $i\not\in G$ and $|G|=i-k$. Moreover, since $u>x_i^i$ in the reverse lexicographic order, it follows that $G\subset [i-1]$. Hence $k>0$, because  $u$ is of degree $i$.
Then  $u\in  \supp(x_i^kp_{i-k})$ with coefficients $1$, because $x_G\in \supp(p_{i-k})$, and we also have that $u\in \supp(x_{i}^{k-1}p_{i-k+1})$ with coefficient $1$, because $x_Gx_i=x_{G\union\{i\}}$ belongs to $\supp(p_{i-k+1})$.

Since $u$  does not belong  to the support of any other summand $x_i^{i-j}p_j$ of $g_i$, and since the summands in which $u$ appears in the support, have different signs, we conclude that $u\not \in \supp(g_i)$, as desired.

In order to deal with the  general case we observe that
\[
S/(I_\Delta, p_1(\Delta),\ldots, p_d(\Delta))=
S/(I_\Delta, p_1(\Gamma_n),\ldots, p_n(\Gamma_n)).
\]
In particular it follows that $\dim S/(I_\Delta, p_1(\Delta),\ldots, p_d(\Delta))=0$. This yields the desired conclusion.
\end{proof}

\begin{Remark}{\em
Let $p_i\in S$ be defined as in (\ref{pi}), and let $J=(p_1,\ldots,p_n)$. Then for the reverse lexicographical order $<$ we have
\[
\ini_<(J)=(x_1,x_2^2,\ldots,x_i^i,\ldots,x_n^n).
\]
Indeed, Theorem~\ref{standard} implies that $p_1, \ldots, p_n$ is a regular sequence, and since $x_1, x_2^2,\ldots,x_n^n$ is also a regular sequence and since  $\deg x_i^i=\deg p_i$ for $i=1,\ldots,n$, we see that
\[
\length(S/J)=\length(S/(x_1,x_2^2,\ldots,x_n^n)\geq \length(S/\ini_<(J))=\length(S/J).
\]
Hence,  $\length(S/(x_1,x_2^2,\ldots,x_n^n)= \length(S/\ini_<(J))$,    which yields the desired
conclusion.}
\end{Remark}

\begin{Corollary}
\label{sufficient}
Let $e$ be the number of facets $F$  of $\Delta$ with $|F|=d$. Then $K[\Delta]$ is Cohen-Macaulay if and only if
\[
\length(S/(I_\Delta, p_1(\Delta),\ldots,p_d(\Delta)))=d!e.
\]
\end{Corollary}

\begin{proof}
We note that $e=e(R)$ is the multiplicity of $K[\Delta]$.
Thus the assertion follows from Theorem~\ref{standard} and Proposition~\ref{graded}.
\end{proof}

Let $\Delta$ be a simplicial complex of dimension $d-1$.  While $d!e$ only depends on the simplicial complex $\Delta$, in general the number $\length(S/(I_\Delta, p_1(\Delta),\ldots,p_d(\Delta)))$  also depends  on the characteristic of the field $K$. This is not  surprising since this is also  the case for the Cohen--Macaulay property of $K[\Delta]$. %We set
%\[
%\lambda(K;\Delta)=\length(S/(I_\Delta, p_1(\Delta),\ldots,p_d(\Delta)))-d!e.
%\]
%By what we proved above, $\lambda(K;\Delta)\geq 0$,  and
%$\lambda(K;\Delta)=0$, if and only if $K[\Delta]$ is Cohen--Macaulay. We call $\lambda(K;\Delta)$ the {\em Cohen--Macaulay defect} of $\Delta$ (with %respect to $K$).

%For example, if $\Delta$ is the canonical triangulation of the projective plane, then $\lambda(K;\Delta)=1$ if $\chara(K)=2$,  and  %$\lambda(K;\Delta)=0$ otherwise.

\medskip
Corollary~\ref{sufficient} can also be used  as  a computational tool to determine the depth of a Stanley--Reisner ring. We demonstrate this with the following example: consider the  chessboard $\P_n$ of size $n\times n$. The set of non-attacking rooks on $\P_n$ is a simplicial complex  which we denote  $\Delta(\P_n)$. Fix a field $K$. For $n>1$, the Stanley Reisner ring $K[\Delta(\P_n)]$ is  not Cohen--Macaulay. %It would be interesting to determine  $\depth K[\Delta(\P_n)]$  and the Cohen--Macaulay defect  $\lambda(K;\Delta(\P_n))$ for all $n$.
Indeed if $n=2$, then $\Delta(\P_n)$ is not connected and hence not Cohen--Macaulay and if $n>2$, then
for any face $F$ with $|F|=n-2$, $\link_{\Delta(\P_n)}(F)=\Delta(\P_2)$. Hence $\Delta(\P_n)$ can not be  Cohen--Macaulay. But what is the depth of $K[\Delta(\P_n)]$?
For $n=2,3$ the depth  can be computed by using  the depth command  implemented in CoCoA. But already for $n=4$,  $\depth K[\Delta(\P_4)]$ can not be computed by CoCoA. Instead we use the  fact, first shown by  Smith \cite[Theorem 4.8]{Sm}, that for any  simplicial complex  $\Delta$ one has
$$
\depth K[\Delta] =\max\{i\: K[\Delta^{(i)}] \text{ is Cohen-Macaulay}\},
$$
where
$
\Delta^{(i)}=\langle F\in \Delta\: |F|=i\rangle
$
is the $i$th skeleton of $\Delta$.

In order to obtain  the Stanley--Reisner ideal of the $(n-1)$-skeleton of $\P_n$ we have to add to $I_{\Delta(\P_n)}$ the monomials corresponding to the facets of $\P_n$.  Then we use  Corollary~\ref{sufficient} to check the Cohen--Macaulayness of the $(n-1)$-skeleton.  For example when $n=4$, the calculation for $I_{\Delta(\P_4)^{(3)}}$ gives
 $$\length(S/(I_{\Delta(\P_4)^{(3)}}, p_1(\Delta(\P_4)^{(3)}),p_2(\Delta(\P_4)^{(3)}), p_3(\Delta(\P_4)^{(3)})))=
6e(K[\Delta(\P_4)^{(3)}]).$$
The length on the left hand side can be computed by means of the  multiplicity command of CoCoA. The output comes almost immediately.
Hence,  $K[\Delta(\P_4)^{(3)}]$ is Cohen--Macaulay and so $\depth K[\Delta(\P_4)]=3$.


\begin{thebibliography}{10}

\bibitem{A}  Alekseev, V.E.: An upper bound for the number of maximal independent sets in a graph. Discrete Math. Appl.   \textbf{17}, No. 4, 355–-359 (2007)

\bibitem{BH}
Bruns, W., Herzog, J.: Cohen--Macaulay rings.
Cambridge studies in advanced mathematics.  \textbf{39}. Cambridge University Press, London, Cambridge, New York (1993)

\bibitem{BH2} Bruns, W., Herzog, J.: Semigroup rings and simplicial complexes. J. Pure Appl. Algebra \textbf{122}, 185–208 (1997)

\bibitem{CCR} Castrill\`{o}n, I.D., Cruz, R., Reyes, E.: On well-covered, vertex decomposable and Cohen-Macaulay graphs. Electron J. Combin. \textbf{23}(2) (2016)

\bibitem{Eis}  Eisenbud, D.: Commutative Algebra: with a view toward algebraic geometry. Vol. 150. Springer Science and Business Media (2013)

%\bibitem{FV} Francisco, C. A.,  Van Tuyl, A.: Sequentially Cohen-Macaulay edge ideals.
%Proc. Amer. Math. Soc \textbf{135}(8), 2327--2337 (2007)

\bibitem{Fro} Fr\"{o}berg, R.:   On Stanley-Reisner rings. Banach Center Publications \textbf{26}, 57--70 (1990)

\bibitem{JH} Herzog, J.: Ein Cohen-Macaulay-Kriterium mit Anwendungen auf den Konormalenmodul und den Differentialmodul. Mathematische Zeitschrift.   \textbf{163}, 149--162 (1978)

\bibitem{HH} Herzog, J., Hibi, T.: Monomial ideals. Graduate Texts in Mathematics. \textbf{260},
Springer-Verlag London, Ltd., London (2011)

\bibitem{HTT}  Hoang, D.T. and   Trung, T. N.: Coverings, matchings and the number of maximal independent sets of graphs. Australasian Journal of Combinatorics. \textbf{73}(3), 424–-431 (2019)

\bibitem{Ho} Hochster, M.: Cohen-Macaulay rings, combinatorics, and simplicial complexes.  Ring Theory II. Lect. Notes in Pure Appl. Math. \textbf{26}, 171--223 (1977)


\bibitem{K} Katzmann, M.: Characteristic-independence of Betti numbers of graph ideals. J. Combin. Theory
Ser. A \textbf{113}, no. 3,  435–-454 (2006)


\bibitem{MMCRY} Mahmoudi, M., Mousivand, A., Crupi, M., Rinaldo, G., Terai, N., and Yassemi, S.:  Vertex decomposability and regularity of very well-covered graphs. Journal of Pure and Applied Algebra. \textbf{215}(10), 2473--2480 (2011)

\bibitem{Se} Serre, J.P.: Alg\`{e}bre locale, multiplicit\'{e}s: cours au Coll\`{e}ge de France, 1957-1958.  \textbf{11},  Springer Science \& Business Media  (1997)

\bibitem{Sm}  Smith D.E.: On the Cohen–Macaulay property in commutative algebra and simplicial topology. Pacific J. Math. \textbf{141},
165--196 (1990)


\bibitem{St1} Stanley, R.P.: Combinatorics and Commutative Algebra. Volume 41, Springer (1996)


\bibitem{SH} Swanson, I., Huneke, C.: Integral closure of ideals, rings, and modules. \textbf{13},  Cambridge University Press  (2006)


\bibitem{Ta} Takayama, Y.: Combinatorial characterizations of generalized Cohen-Macaulay monomial ideals. Bulletin mathématique de la Société des Sciences Mathématiques de Roumanie,   327--344 (2005)





\end{thebibliography}
\end{document}